\documentclass[a4paper,11pt]{amsart}

\makeatletter
 
  \@addtoreset{equation}{subsection}
\makeatother

\usepackage{amssymb}
\usepackage{amsmath}

\theoremstyle{plain}
\newtheorem{theorem}{Theorem}[section]

\newtheorem{lemma}[theorem]{Lemma}
\newtheorem{proposition}[theorem]{Proposition}
\theoremstyle{definition}
\newtheorem{definition}[theorem]{Definition}

\newtheorem{remark}[theorem]{Remark}

\newtheorem{example}[theorem]{Example}

\newcommand{\Z}{\mathbb{Z}}
\newcommand{\Q}{\mathbb{Q}}

\newcommand{\G}{\mathbb{G}}

\newcommand{\pd}{\partial}

\newcommand{\Char}{\operatorname{Char}}

\newcommand{\Tor}{\operatorname{Tor}}

\newcommand{\Gal}{\operatorname{Gal}}
\newcommand{\ord}{\operatorname{ord}}
\newcommand{\Br}{\operatorname{Br}}
\newcommand{\Pic}{\operatorname{Pic}}

\newcommand{\Hom}{\operatorname{Hom}}

\newcommand{\Spec}{\operatorname{Spec}}

\newcommand{\Image}{\operatorname{Im}}
\newcommand{\coker}{\operatorname{coker}}

\newcommand{\Div}{\operatorname{Div}}
\newcommand{\divi}{\operatorname{div}}
\newcommand{\Zar}{\operatorname{Zar}}

\newcommand{\cd}{\operatorname{cd}}

\DeclareMathOperator{\Nis}{Nis}
\DeclareMathOperator{\et}{et}
\DeclareMathOperator{\eff}{eff}
\DeclareMathOperator{\DM}{DM}
\DeclareMathOperator{\DMn}{\DM^{\eff,-}_{\Nis}}

\begin{document}


\title{The Brauer-Manin pairing, class field theory \\
and motivic homology}



\author{Takao Yamazaki}
\date{\today}
\address{Mathematical Institute, Tohoku University,
  Aoba, Sendai 980-8578, Japan}
\email{ytakao@math.tohoku.ac.jp}

\begin{abstract}
For a smooth proper variety over a $p$-adic field,
its Brauer group and abelian fundamental group
are related to higher Chow groups
by the Brauer-Manin pairing and class field theory.
We generalize this relation to 
smooth (possibly non-proper) varieties,
using motivic homology and
a variant of Wiesend's ideal class group.
Several examples are discussed.
\end{abstract}



\maketitle


\section{Introduction}
\subsection{The Brauer-Manin pairing and class field theory}\label{intro-1}
Let $p$ be a prime number, and
$k$ a finite extension of $\Q_p$.
Let $X$ be a smooth variety over $k$.
For $i \in \Z_{\geq 0}$,
we write $X_{(i)}$ for the set of
all points of $X$ of dimension $i$.
Let $\Br(X)$ be the cohomological Brauer group of $X$,
and let $\pi_1^{ab}(X)$ be the abelian etale fundamental group of $X$.
For any $x \in X_{(0)}$, local class field theory yields
canonical maps
\[ \psi_x^* : \Br(x) \cong \Q/\Z, \qquad
   \rho_x : k(x)^* \to \pi_1^{ab}(x).
\]
As for the first map, 
it is convenient for our purpose
to consider its dual.
Putting $A^* := \Hom(A, \Q/\Z)$ for an abelian group $A$,
we define $\psi_x$ to be the composition of
the dual of $\psi_x^*$ and the canonical inclusion
\[ \psi_x : \Z \hookrightarrow \hat{\Z} \cong \Br(x)^*. \]
Since both of $\Br(-)^*$ and $\pi_1^{ab}(-)$ are covariant functorial,
we get homomorphisms
\begin{equation}\label{twomaps}
 \tilde{\psi}_X : 
 Z_0(X) := \bigoplus_{x \in X_{(0)}} \Z \to \Br(X)^*,
\qquad
   \tilde{\rho}_X : Z_0^1(X) := \bigoplus_{x \in X_{(0)}} k(x)^*
\to \pi_1^{ab}(X)
\end{equation}
by taking the direct sum of the $\psi_x$'s and $\rho_x$'s.

When $X$ is proper over $k$,
Manin \cite{manin} and Bloch/Saito \cite{bloch2, saito1} observed that
$\tilde{\psi}_X$ and $\tilde{\rho}_X$ 
factor respectively through
\begin{align*}
CH_0(X) &:= \coker[\bigoplus_{y \in X_{(1)}} k(y)^* \to Z_0(X)],
\\
SK_1(X) &:= \coker[\bigoplus_{y \in X_{(1)}} K_2 k(y) \to Z_0^1(X)];
\end{align*}
the induced 
pairing 
$CH_0(X) \times \Br(X) \to \Q/\Z$
and the induced map
$SK_1(X) \to \pi_1^{ab}(X)$
are called the 
{\it Brauer-Manin pairing} and 
{\it reciprocity map of class field theory} 
respectively.
Both are studied intensively by several authors;
for the Brauer-Manin pairing, see
\cite{ps, ct-survey, ct-survey2, cts, ss, yama};
for the reciprocity map, see
\cite{sz2, sz, js, js2, sato2, yama2}.

If $X$ is not proper over $k$, however,
$\tilde{\psi}_X$ and $\tilde{\rho}_X$ 
do not factor through 
$CH_0(X)$ or $SK_1(X)$.
To this end,
we shall introduce good quotients 
of $Z_0(X)$ and $Z_0^1(X)$.

\subsection{Wiesend's tame ideal class group}\label{sect:intro-wie}
Let $V$ be a variety over a field $F$.
Take $y \in V_{(1)}$.
Let $C(y)$ be the closure of $\{ y \}$ in $V$,
$\tilde{C}(y) \twoheadrightarrow C(y)$ the normalization,
$\tilde{C}(y) \hookrightarrow \bar{C}(y)$ the smooth completion,
and $C_{\infty}(y) = \bar{C}(y) \setminus \tilde{C}(y)$.
For $x \in C_{\infty}(y)$,
we take a uniformizer $\pi_x \in F(y)^*$ at $x$.
We define for $r \in \Z_{>0}$
\[ UK_r^M F(y) := 
\ker[ K_r^M F(y) \to \bigoplus_{x \in C_{\infty}(y)} 
(K_{r-1}^M F(x) \oplus K_r^M F(x))].
\]
Here the $x$-component of the map is defined 
by $a \mapsto (\pd_x(a), \pd_x( \{ \pi_x \} \cup  a ))$
for $a \in K_r^M F(y)$,
where $\pd_x$ is the tame symbol at $x$.
This group does not depend on the choice of $\pi_x$.
If $V$ is proper over $F$,
then $UK_r^M F(y) = K^M_r F(y).$
For example, 
$UK_1^M F(y)$ is the group of rational functions on $\bar{C}(y)$
which takes value $1$ at all points of $C_{\infty}(y)$.
The following definition is a 
natural outcome of an idea of Wiesend \cite{wiesend}.

\begin{definition}\label{classgroup}
Let $V$ be a variety over a field $F$, 
and let $r \in \Z_{\geq 0}$.
We define 
{\it Wiesend's tame ideal class group} of degree $r$
to be
\[ C_r(V) := \coker[ 
\bigoplus_{y \in V_{(1)}} UK_{r+1}^M F(y) \hookrightarrow
\bigoplus_{y \in V_{(1)}} K_{r+1}^M F(y) \overset{(*)}{\to}
\bigoplus_{x \in V_{(0)}} K_{r}^M F(x)],
\]
where $(*)$ is the boundary map 
of Gersten complex of the Milnor $K$-sheaf.
(In particular, its image is in the direct sum.)
\end{definition}
By definition,
$C_0(V)$ and $C_1(V)$ are quotients of
$Z_0(V)$ and $Z_0^1(V)$ respectively.
If $V$ is proper over $F$, then 
we have $C_0(V) = CH_0(V)$ and $C_1(V) = SK_1(V)$.

\begin{remark}\label{cg-rem}
Suppose that $V=\bar{C} \setminus C_{\infty}$,
where $\bar{C}$ is a smooth projective 
geometrically irreducible curve over $F$
and $C_{\infty}$ is a closed reduced subvariety of $\bar{C}$.
Then our definition of $C_0(V)$ coincides with
that of the {\it group of classes of divisors on $\bar{C}$
prime to $C_{\infty}$ modulo $C_{\infty}$-equivalence}
defined in \cite[Chap. V, \S 2]{serre3}.
It is proved \cite[Chap. V, Thm. 1]{serre3} that
$\ker(\deg :C_0(V) \to \Z)$
is represented by a semi-abelian variety $J$,
which is called the {\it generalized Jacobian variety}
of $\bar{C}$ with modulus $C_{\infty}$.
(Here the degree map $\deg : C_0(V) \to \Z$
is induced by the usual degree map $Z_0(V) \to \Z$.)
For simplicity, we call $J$
``the generalized Jacobian of $V$''.
One can also interpret $C_0(V)$
as the Picard group of the singular curve
obtained from $\bar{C}$ by 
contracting all points of $C_{\infty}$
to one point
\cite[Chap. V, \S 2 and Chap. IV. \S 4]{serre3}.
On the other hand,
$C_0(V)$ is also isomorphic to
the relative Picard group
$\Pic(\bar{C}, C_{\infty})$
(see \cite[Theorem 3.1]{sv}).
\end{remark}

\subsection{Motivic homology}
Let $V$ be a smooth variety over a perfect field $F$,
and let $i, j \in \Z$.
Its motivic homology 
$H_i^M(V, \Z(j))$ 
is defined to be the group of homomorphisms
$\Hom_{\DMn(F)}(\Z(j)[i], M(V))$
in Voevodsky's category
(cf. \cite{fv}, \cite[(14.17)]{mvw}).
When $j=0$, the group $H_i^M(V, \Z(0))$ agrees with
{\it Suslin's algebraic singular homology} $h_i(V)$ \cite{sv},
and admits a (relatively simple)
description in terms of algebraic cycles
(cf. \cite[(3.2.7)]{voevodsky}, \cite[(14.18)]{mvw}).
There is a similar but complicated description also for $j \not= 0$
(cf. \cite[(9.4)]{fv}).
The following theorem, which plays a crucial role in our paper,
provides a simpler description in a special case.

\begin{theorem}\label{cg-h}
Let $V$ be a smooth variety over a perfect field $F$,
and let $r \in \Z_{\geq 0}$.
Then there is a canonical isomorphism
\begin{equation}\label{ch-h-eq}
C_r(V) \cong H_{-r}^M(V, \Z(-r)).
\end{equation}
\end{theorem}

When $r=0$, 
using the comparison of $H_{0}^M(V, \Z(0))$
with $h_0(V)$ recalled above,
this follows from a result of
Schmidt (\cite{schmidt} Theorem 5.1).
We shall prove this theorem by reducing to the case $r=0$
in \S \ref{sect-mot}.
In what follows, we often identify
$C_r(V)$ and $H_{-r}^M(V, \Z(-r))$.

\begin{remark}\label{geisserremark}
For a smooth projective variety $V$ 
over a {\it finite} field,
unramified class field theory \cite{ks}
relates $CH_0(V)$ with
the abelian fundamental group of $V$.
This has been generalized
by Schmidt and Spiess \cite{ss2}
to smooth (possibly non-proper) variety $V$,
in which $CH_0(V)$ is replaced by
Suslin's algebraic singular homology $h_0(V)$.
(Note that $h_0(V) \cong C_0(V)$ 
by Schmidt's theorem mentioned above.
See also recent works of Geisser \cite{geisser1, geisser2}
for a further generalization.)
The basic strategy of our paper
is to follow their argument over a $p$-adic base field.
\end{remark}

\subsection{Open varieties over a $p$-adic field}
We go back to the situation of \S \ref{intro-1}.
Schmidt and Spiess constructed a map
connecting motivic homology and 
etale cohomology with compact support,
which we will recall in \S \ref{sect-et}.
As an application,
we deduce the following proposition in \S \ref{sect-var}.

\begin{proposition}\label{well-def}
Let $X$ be a smooth variety over a finite extension $k$ of $\Q_p$.
The homomorphisms \eqref{twomaps} 
induce well-defined homomorphisms
\begin{equation}\label{bmcft}
\psi_X : C_0(X) \to \Br(X)^*,
\qquad
\rho_X : C_1(X) \to \pi_1^{ab}(X).
\end{equation}
\end{proposition}

In \S \ref{sect-cv}, we consider the case where $X$ is a curve.
In this case, the map $\psi_X$ was already studied by
Scheiderer-van Hamel \cite{svh}
(see Theorem \ref{thm-svh} below),
and the map $\rho_X$ is closely related to 
work of Hiranouchi \cite{hira} 
(see Theorem \ref {thm-hira} and Remark \ref{rem-hira} below).
When $X$ is of dimension two or higher,
the maps $\psi_X$ and $\rho_X$ 
are not very close to an isomorphism
even if $X$ is projective over $k$.
We study several examples of surfaces in \S \ref{sect-surface}.
As a sample, here we mention the following result.
(We write $\bar{V} = V \times_k \bar{k}$
for a variety $V$ over a field $k$ with an algebraic closure $\bar{k}$.)
\begin{theorem}\label{rational0}
Let $X$ be a smooth projective geometrically irreducible surface 
over a finite extension $k$ of $\Q_p$.
Suppose that $\bar{X}$ is {\it rational}.
Let $U$ be an open subvariety of $X$.
\begin{enumerate}
\item
Suppose that the irreducible components of
$\bar{X} \setminus \bar{U}$ generate 
the N\'eron-Severi group $NS(\bar{X})$ of $\bar{X}$.
Then,
the kernel of $\psi_U$ is the maximal divisible subgroup of $C_0(U)$.
However, there is an example of $U$ such that
$\ker(\psi_U \otimes \Z/n\Z) \not= 0$ for 
all $n \in \Z_{>0}$ divisible by some (fixed) integer $N$.
\item
The kernel of $\rho_U$ is the maximal divisible subgroup of $C_1(U)$,
and $\rho_U \otimes \Z/n\Z$ is bijective
for all $n \in \Z_{>0}$.
\end{enumerate}
\end{theorem}

Note that, if $X$ is a (projective smooth) rational surface,
then $\ker(\psi_X) = 0$
and $\ker(\psi_X \otimes \Z/n\Z) = 0$ for all sufficiently divisible $n$
(cf. Theorem \ref{rat}).
Note also that there are examples of (non-rational)
projective smooth surfaces $X$ and $X'$ for which
$\ker(\psi_X)$ and $\ker(\rho_{X'})$ are not divisible 
(cf. \cite{ps, sato2}).

\subsection{Conventions}\label{convention}
Let $A$ be an abelian group.
For a non-zero integer $n$,
we write $A[n]$ and $A/n$ for the kernel and cokernel
of the map $n: A \to A.$
This notation is sometimes used when $n=\infty$,
in which case we mean 
$A[n] = A_{\Tor}$ is the subgroup of torsion elements in $A$,
and
$A/n = A \otimes \Q/\Z$.
We write $A_{\Q}$ for $A \otimes_{\Z} \Q$.
We define $A_{\Div} := \Image[\Hom(\Q, A) \to \Hom(\Z, A)=A]$
to be the maximal divisible subgroup in $A,$
and $A_{\divi} := \cap_{n \in \Z_{>0}} nA$
the subgroup of divisible elements in $A$.
Note that we always have $A_{\Div} \subset A_{\divi}$,
and that $A_{\Div} = A_{\divi}$ holds 
if $A[n]$ is finite for all $n \in \Z_{>0}$.

Let $f:A \to B$ be a homomorphism of abelian groups.
We write $f/n$ for $f \otimes \Z/n$ when $n \in \Z_{>0}$,
and for $f \otimes \Q/\Z$ when $n=\infty$.
Let $m, n \in \Z_{>0}$.
The map $\ker(f/n) \to \ker(f/nm)$
induced by the map $m : A \to A$ is called the canonical map,
so that $\{ \ker(f/n) \}_n$ becomes an inductive system
whose limit is $\ker(f \otimes \Q/\Z)$.
The map $\ker(f/nm) \to \ker(f/n)$
induced by the identity map on $A$ is called the canonical map,
so that $\{ \ker(f/n) \}_n$ becomes an inverse system.

Let $F$ be a field.
A separated scheme of finite type over $F$ is called a variety over $F$.
Let $X$ be a variety over $F$.
For $i \in \Z_{\geq 0}$,
we write $X_{(i)}$ and $X^{(i)}$ for the set of
all points on $X$ of dimension $i$ and of codimension $i$
respectively.
We write $\bar{F}$ for an algebraic closure of $F$
and $\bar{X}$ for the base change $X \times_{\Spec F} \Spec \bar{F}$.
A closed subset of $X$ is always regarded 
as a {\it reduced} subvariety of $X$.

\noindent
{\it Acknowledgements.}
The author express his gratitude to Professors Shuji Saito 
and Bruno Kahn
for helpful comments.
Thanks are also due to the referee 
for his instrumental comments.
This research is partially supported by 
Grand-in-Aid for Young Scientists (A) (22684001).

\section{Wiesend's ideal class group and motivic homology}\label{sect-mot}
Let $F$ be a perfect field.
Except in the beginning of \S \ref{first} and in \S \ref{subsec-cg},
we will assume the characteristic of $F$ is zero.

\subsection{Motivic homology and cohomology}\label{first}
We recall some facts from \cite{voevodsky} and \cite{mvw}. 
Let $\DMn(F)$ denote 
the rigid triangulated tensor category of effective motivic complexes 
(\cite{voevodsky} 3.1, \cite{mvw} 14.1). 
There is a functor $M$ from the category of varieties over $F$
to $\DMn(F)$.
For a variety $X$ and $n \in \Z_{>0} \cup \{ \infty \}$,
we write 
$M(X, \Z/n\Z) = M(X) \otimes^{\mathbb{L}} \Z/n\Z$
(cf. \S \ref{convention}).
Let $i, j \in \Z$.
Motivic cohomology
and motivic homology
of a variety $X$ 
are defined by
\begin{align*}
H^i_M&(X, \Z(j)) = \Hom_{\DMn(F)}(M(X), \Z(j)[i]),
\\
H_i^M&(X, \Z(j)) = \Hom_{\DMn(F)}(\Z(j)[i], M(X)).
\end{align*}
For $n \in \Z_{>0} \cup \{ \infty \}$,
their coefficient version are defined by
\begin{align*}
H^i_M&(X, \Z/n \Z(j)) = \Hom_{\DMn(k)}(M(X), \Z/n \Z(j)[i]),
\\
H_i^M&(X, \Z/n \Z(j)) = \Hom_{\DMn(k)}(\Z(j)[i], M(X, \Z/n \Z)).
\end{align*}
They fit into (obvious) exact sequences
\begin{align}
\label{univcoeff0}
 0 \to 
H^i_M(X, \Z(j))/n \to
H^i_M(X, \Z/n\Z(j)) \to
H^{i+1}_M(X, \Z(j))[n]
\to 0,
\\
\label{univcoeff}
 0 \to 
H_i^M(X, \Z(j))/n \to
H_i^M(X, \Z/n\Z(j)) \to
H_{i-1}^M(X, \Z(j))[n]
\to 0.
\end{align}
Motivic cohomology 
$H^i_M(X, \Z(j))$
is contravariantly functorial in $X$.
It also has covariant functionality:
if $f: X \to Y$ is 
a proper flat equidimensional morphism of relative dimension $d$,
then there is an induced map
$H^i_M(X, \Z(j)) \to H_M^{i-2d}(Y, \Z(j-d))$.
Motivic homology 
$H_i^M(X, \Z(j))$
is covariantly functorial in $X$.
It also has contravariant functionality:
if $f: X \to Y$ is 
a proper flat equidimensional morphism of relative dimension $d$,
then there is an induced map
$H_i^M(Y, \Z(j)) \to H^M_{i+2d}(X, \Z(j+d))$.

The following fact will play an important role 
in the proof of Theorem \ref{cg-h}:
for any variety $X$, 
there is a decomposition 
\begin{equation}\label{dec}
H^M_i(X \times \G_m, \Z(j)) \cong
H^M_{i-1}(X, \Z(j-1)) \oplus  H^M_{i}(X, \Z(j)).
\end{equation}
This is deduced from
Mayer-Vietoris sequence, projective bundle formula
and $\mathbb{A}^1$-homotopy invariance.

From now through the end of \S \ref{motcpx},
we assume the characteristic of $F$ is zero.
If $X$ is a smooth variety,
we have a canonical isomorphism 
\begin{equation}\label{ch-hc}
H^i_{M}(X, \Z(j)) \cong CH^j(X, 2j-i),
\end{equation}
where the right hand side is Bloch's higher Chow group.
If $X$ is a smooth projective variety of pure dimension $d$,
we also have
\begin{equation}\label{h-ch}
H_i^{M}(X, \Z(j)) 
\cong H^{2d-i}_{M}(X, \Z(d-j))
\cong CH^{d-j}(X, i-2j).
\end{equation}
In particular, if $X = \Spec F$,
then we have for any $r \in \Z_{\geq 0}$
\begin{equation}\label{mcandk}
 H_{-r}^M(\Spec F, \Z(-r)) \cong H^r_M(\Spec F, \Z(r))
   \cong CH^r(\Spec F, r) \cong K_r^M F.
\end{equation}
Let $Z$ be a closed subvariety of a smooth variety $X$.
Suppose $Z$ is smooth and of pure codimension $c$.
Then we have a long exact sequence
\begin{equation}\label{gysin}
\cdots \to 
H_{i+1}^M(X, \Z(j)) \to 
H_{i+1-2c}^M(Z, \Z(j-c)) \to 
H_{i}^M(X \setminus Z, \Z(j)) \to
H_{i}^M(X, \Z(j)) \to 
\cdots.
\end{equation}

\begin{lemma}\label{bydimension0}
Let $X$ be a smooth variety of pure dimension $d$,
and let $i, j \in \Z$.
\begin{enumerate}
\item
Suppose either $j<0,~ i > j+d$ or $i > 2j$. 
Then, we have $H^i_{M}(X, \Z(j))=0$.
\item
Suppose $j>i$.
Then, we have $H_i^{M}(X, \Z(j))=0$.
\end{enumerate}
The same holds for the $\Z/n$-coefficient version 
for any $n \in \Z_{>0} \cup \{ \infty \}$.
\end{lemma}
\begin{proof}
For (1), see \cite[(3.6), (19.3)]{mvw}.
If $X$ is projective,
(2) follows from (1) and \eqref{h-ch}.
The general case follows 
by induction on $\dim X$ and \eqref{gysin}.
\end{proof}

By Lemma \ref{bydimension0} and \eqref{univcoeff},
we can identify
$H_i^M(X, \Z(i))/n = H_i^M(X, \Z/n(i))$
for any $i \in \Z$ and $n \in \Z_{>0}$,
which will be frequently used without further notice.

\begin{remark}
If $X$ is a smooth variety of pure dimension $d$,
then motivic homology $H_i^M(X, \Z(j))$ is
isomorphic to {\it motivic cohomology with compact support}
$H^{2d-i}_{M, c}(X, \Z(d-j))$.
However, the work of Geisser 
(cf. Remark \ref{geisserremark})
suggests that motivic homology would be better
for a further generalization,
so we opt to use homology theory.
\end{remark}

\subsection{Motivic complex}\label{motcpx}
We recall some results on motivic complex.
Let $X$ be a smooth variety over $F$, and let $j \in \Z_{\geq 0}$.
There is a complex $\Z(j)_X$
of Zariski sheaves on $X$ \cite[(3.1)]{mvw},
which is concentrated in degrees $\leq j$.
The hypercohomology of $\Z(j)_X$ agrees with motivic cohomology
\cite[(14.16)]{mvw}:
\[ H^i_{\Zar}(X, \Z(j)_X) \cong H^i_M(X, \Z(j)). \]

Thanks to the recent resolution of Bloch-Kato conjecture
by Rost and Voevodsky (see \cite{sj, voevodsky2, weibel}),
the following important result
of Suslin-Voevodsky \cite{sv2} 
(see also Geisser-Levine \cite{gl})
holds unconditionally.

\begin{theorem}[\cite{sv2, gl}]\label{bigth}
Let $X$ be a smooth variety over $F$, 
and let $\pi : X_{et} \to X_{\Zar}$ be the natural map of sites.
Let $j \in \Z_{\geq 0}$ and $n \in \Z_{>0}$.
\begin{enumerate}
\item
There is a canonical isomorphism
\begin{equation}\label{mapinone}
\pi^* \Z(j)_X \otimes^{\mathbb{L}} \Z/n \cong \mu_n^{\otimes j}.
\end{equation}
Consequently, we have a canonical map for any $i \in \Z_{\geq 0}$
\begin{equation}\label{cetale}
H^i_M(X, \Z/n(j)) \to H^i_{et}(X, \mu_n^{\otimes j}).
\end{equation}
\item
The map \eqref{mapinone} induces an isomorphism
\[ \Z(j)_X \otimes^{\mathbb{L}} \Z/n 
\to \tau_{\leq j} R \pi_* \mu_n^{\otimes j}. \]
Consequently, the map \eqref{cetale} 
is an isomorphism if 
either $i \leq j$ or $j \geq \dim X + \cd F$.
(Here $\cd$ means the cohomological dimension.)
It is an injection if $i=j+1$.
\end{enumerate}
\end{theorem}
In what follows, 
we frequently write
$H^i_{et}(X, \Z/n(j))$ for $H^i_{et}(X, \mu_n^{\otimes j})$.

\subsection{Wiesend's ideal class group}\label{subsec-cg}
In this subsection, we do not assume $\Char F=0$.
We shall prove Theorem \ref{cg-h},
after introducing a simple lemma.

\begin{lemma}
Let $r \in \Z_{\geq 0}$.
Let $f: X \to X'$ be a morphism of varieties over $F$.
There is a unique homomorphism $f_* : C_r(X) \to C_r(X')$
characterized by the following property:
for any $x \in X_{(0)}$, the diagram
\[
\begin{matrix}
K_r^M F(x) & \to & C_r(X)
\\
\qquad \quad
\downarrow^{N_{F(x)/F(f(x))}}  & & \downarrow_{f_*}
\\
K_r^M F(f(x)) & \to & C_r(X')
\end{matrix}
\]
commutes.
(Here the upper and lower horizontal maps are the natural map
to the $x$-component and $f(x)$-component respectively.)
This makes $C_r$ a covariant functor 
on the category of varieties over $F$.
\end{lemma}

\begin{proof}
Uniqueness is clear by the definition of $C_r(X)$.
Functoriality follows from that of the norm map
of Milnor $K$-groups.
We prove well-definedness.
Take $y \in X_{(1)}$ and put $y'=f(y)$.
We need to show that the image of
$UK_{r+1}^M F(y)$ in $C_r(X')$ is trivial.
We consider two cases separately.
First we assume $y' \in X'_{(0)}$.
With the notation of \S \ref{sect:intro-wie},
we regard $\bar{C}(y)$ as a curve over $F(y')$.
Weil reciprocity \cite[Proposition 7.4.4]{gs}
shows that the composition map
\[ K_{r+1}^M F(y) \overset{\oplus \pd_x}{\to} 
\bigoplus_{x \in \bar{C}(y)_{(0)}} K_r^M F(x) 
\overset{\oplus N_{F(x)/F(y')}}{\to} K_r^M F(y') 
\]
is the zero-map.
In view of the definition of $UK_{r+1}^M F(y)$,
this proves our assertion in this case.
Next, we assume $y' \in X'_{(1)}$
so that $f$ induces a morphism
$f_y : \bar{C}(y) \to \bar{C}(y')$ such that
$f_y^{-1}(C_{\infty}(y')) \subset C_{\infty}(y)$.
It suffices to show that
the image of $UK_{r+1}^M F(y)$ by the norm map
$K_{r+1}^M F(y) \to K_{r+1}^M F(y')$ is
contained in $UK_{r+1}^M F(y')$.
This follows from
the following basic fact on the tame symbol.
(We omit its proof.)
\end{proof}

\begin{lemma}\label{basictame}
Let $L/K$ be a finite extension of fields.
Let $v$ a discrete valuation on $K$,
and let $\{ v_1, \cdots, v_f \}$ be 
the set of all extensions of $v$ to $L$.
We write $C$ and $D_i ~(1 \leq i \leq f)$
for the residue field of $v$ and $v_i$ respectively.
Let $r \in \Z_{>0}$.
For any $\lambda \in K_r^M L$,
we have 
\[ \pd_v N_{L/K} \lambda = \sum_{i=1}^f N_{D_i/C} \pd_{v_i} \lambda.
\]
\end{lemma}

\noindent
{\it Proof of Theorem \ref{cg-h}.}
The case $r=0$ is 
a conjunction of 
the comparison theorem of motivic homology with Suslin's homology
(\cite[Corollary 3.2.7]{voevodsky}, \cite[Proposition 14.18]{mvw})
and the comparison theorem of Suslin's homology
and $C_0(V)$ \cite[Theorem 5.1]{schmidt}.
We proceed by induction on $r$.

We construct a commutative diagram
\begin{equation}\label{indfirst}
\begin{matrix}
\bigoplus_{y \in (V \times \G_m)_{(1)}} UK_{r+1}^M F(y) &
\overset{\pd}{\to} &
\bigoplus_{x \in (V \times \G_m)_{(0)}} K_{r}^M F(x) &
\\
\downarrow_f &
&
\downarrow_g &
\\
\bigoplus_{w \in V_{(1)}} UK_{r+2}^M F(w) &
\overset{\pd}{\to} &
\bigoplus_{z \in V_{(0)}} K_{r+1}^M F(z). &
\end{matrix}
\end{equation}
The two horizontal maps $\pd$
are given by the tame symbol.
For $x \in (V \times \G_m)_{(0)}$ and $z \in V_{(0)}$,
the $(x, z)$-component of $g$ is given as follows.
It is the zero-map if $z \not= p_0(x)$,
where $p_0 : V \times \G_m \to V$ is the projection.
Suppose $z=p_0(x)$, so that $F(x)$ is a finite extension of $F(z)$.
The composition $x \to V \times \G_m \to \G_m$
defines an element $\xi(x) \in \G_m(F(x))=F(x)^*$.
Now the map in question is given by
\[  K_r^M F(x) \overset{\text{mult. by}~ \xi(x)}{\to}
    K_{r+1}^M F(x) \overset{N_{F(x)/F(z)}}{\to} K_{r+1}^M F(z).
\]
Next, let $y \in (V \times \G_m)_{(1)}$ and $w \in V_{(1)}$.
If $w \not= p_0(y)$, then
the $(y, w)$-component of $f$ is the zero-map.
Suppose $w=p_0(y)$, so that $F(y)$ is a finite extension of $F(w)$.
The composition $y \to V \times \G_m \to \G_m$
defines an element $\xi(y) \in \G_m(F(y))=F(y)^*$.
Now the $(y, w)$-component of $f$ is given by
\[  UK_{r+1}^M F(y) \overset{\text{mult. by}~ \xi(y)}{\to}
    UK_{r+2}^M F(y) \overset{N_{F(y)/F(w)}}{\to} UK_{r+2}^M F(w).
\]
We check the commutativity.
Let $y \in (V \times \G_m)_{(1)}$ 
and $\lambda \in UK_{r+1}^M F(y)$.
Put $w = p_0(y)$.
When $w \in V_{(0)}$, 
by definition we have $\pd(f(\lambda))=0$,
and setting $z=w$ we have
\begin{align*}
g(\pd(\lambda)) &= g( \sum_{x \in C(y)} \pd_x \lambda)
 = \sum_{x \in \tilde{C}(y)} N_{F(x)/F(z)} \{ \pd_{x} \lambda, \xi(x) \}
\\
 &\overset{(1)}{=} \sum_{x \in \tilde{C}(y)} 
   N_{F(x)/F(z)} \pd_x \{ \lambda, \xi(y) \}
\\
 &\overset{(2)}{=} [\sum_{x \in \tilde{C}(y)} + \sum_{x \in C_{\infty}(y)}]  
     ~(N_{F(x)/F(z)} \pd_x \{ \lambda, \xi(y) \})
\\
 &\overset{(3)}{=}  0.
\end{align*}
Here, at (1) we used the fact that
for all $x \in \tilde{C}(y)$ we have
$\ord_x \xi(y) =0$ and
$\pd_x \{ \xi(y), \pi_x \} = \xi(x)$,
where $\pi_x \in k(y)^*$ is a uniformizer at $x$.
At (2), we used the definition of $UK_{r+1}^M F(y)$
(that is,
$\pd_x(\lambda)=0$ and $\pd_x \{ \lambda, \pi_x \} = 0$
for all $x \in C_{\infty}(y)$).
The equality (3) is Weil reciprocity
\cite[Proposition 7.4.4]{gs}.
Now we suppose $w \in V_{(1)}$.
We take $z \in V_{(0)}$,
and we write $\{x_1, \cdots, x_n\}$ for the set of
all points on $\tilde{C}(y)$ above $z$.
We have 
\begin{align*}
\text{$z$-component of}~
\pd(f(\lambda)) &= \pd_z N_{F(y)/F(w)} \{ \lambda, \xi(y) \}
\\
&\overset{(1)}{=} 
\sum_{i=1}^n N_{F(x_i)/F(z)} \pd_{x_i} \{ \lambda, \xi(y) \}
\\
&\overset{(2)}{=} 
\sum_{i=1}^n  N_{F(x_i)/F(z)} \{\pd_{x_i}(\lambda), \xi(x_i)\}
\\
&= \text{$z$-component of}~g(\pd(\lambda)).
\end{align*}
At (1), we used Lemma \ref{basictame}.
The equality (2) follows from the fact that
$\ord_x \xi(y) = 0$ and 
$\pd_x \{ \xi(y), \pi_x \} = \xi(x)$
for all $x \in \tilde{C}(y)$ 
This proves the commutativity.


It can be seen that both $f$ and $g$ are surjective without difficulty.
As a consequence,
we get a surjection
\[ \alpha : C_r(V \times \G_m) \to C_{r+1}(V). \]
It follows from the definition that
the composition
$C_r(V) = C_r(V \times \{ 1 \}) 
\to C_r(V \times \G_m) \overset{\alpha}{\to} C_{r+1}(V)$
is the zero-map.

Next, we consider a diagram
(cf. Kahn \cite[(4.3)]{kahn})
\begin{equation}\label{doubleind}
\begin{matrix}
\bigoplus_{x \in (V \times \G_m)_{(0)}} K_{r}^M F(x) &
\to &
H_{-r}^M(V \times \G_m, \Z(-r))
\\
\downarrow^g &
&
\downarrow &
\\
\bigoplus_{z \in V_{(0)}} K_{r+1}^M F(z) &
\to &
H_{-r-1}^M(V, \Z(-r-1)).
\end{matrix}
\end{equation}
Here the horizontal maps are induced by the functoriality
of motivic homology 
and \eqref{mcandk}.
The right vertical map is the projection with respect to 
the decomposition \eqref{dec}
The left vertical map $g$ was defined above.

We denote by $P_r(V)$ the assertion
that the diagram \eqref{doubleind} is commutative.
We also denote by $Q_r(V)$ the assertion that
$\oplus_{x \in V_{(0)}} K_{r}^M F(x) \to H_{-r}^M(V, \Z(-r))$
induces an isomorphism
$C_r(V) \cong H_{-r}^M(V, \Z(-r))$
(which is what the theorem claims).
We claim:
\begin{enumerate}
\item We have $Q_0(V)$ for all $V$.
\item Fix $r \in \Z_{\geq 0}$ 
and a finite extension $F'/F$.
If $Q_r(\G_m \times \Spec F')$ holds,
then $P_r(\Spec F')$ holds.
\item Fix $r \in \Z_{\geq 0}$.
If $P_r(\Spec F')$ holds for any finite extension $F'/F$,
then $P_r(V)$ holds for all $V$.
\item Fix $r \in \Z_{\geq 0}$ and $V$.
If $P_r(V), Q_r(V)$ and $Q_r(V \times \G_m)$ hold,
then $Q_{r+1}(V)$ holds.
\end{enumerate}
Indeed, (1) was already remarked at the beginning of the proof.
(2) holds because 
we have $Q_{r+1}(\Spec F')$ by \eqref{mcandk}.
(By the definition of $C_r(V)$,
if one has
$Q_r(V \times \G_m)$ and $Q_{r+1}(V)$,
then $P_r(V)$ becomes trivial.)
(3) is clear from the definition of $P_r(V)$.
We prove (4). 
By $P_r(V), Q_r(V \times \G_m)$ and the surjectivity of $f$,
we get a well-defined map $(!)$
in the commutative diagram
\[
\begin{matrix}
C_r(V \times \G_m)/C_r(V) &
\overset{\alpha}{\twoheadrightarrow} &
C_{r+1}(V)
\\
\downarrow &
&
\downarrow^{(!)}
\\
H_{-r}^M(V \times \G_m, \Z(-r))/H_{-r}^M(V, \Z(-r)) &
\cong &
H_{-r-1}^M(V, \Z(-r-1)).
\end{matrix}
\]
By $Q_r(V)$ and $Q_r(V \times \G_m)$,
the left vertical map is an isomorphism.
It follows the right vertical map is also an isomorphism,
which shows (4).
Now the theorem follows by induction.
\qed

\section{Etale cohomology with compact support}\label{sect-et}
In this section,
$F$ is a field of characteristic zero.

\subsection{The map $c^{i, j}_{X, n}$}
Let $V$ be a smooth irreducible variety over $F$
of dimension $d$.
Schmidt and Spiess 
\cite[Proposition 3.1 and p. 26 Remarks]{ss2}
constructed a canonical homomorphism
\begin{equation}\label{ss-map}
c^{i, j}_{V, n} : 
 H_{2d-i}^M(V, \Z/n(d-j)) \to H^i_{et, c}(V, \Z/n(j))
\end{equation}
for any $i, j \in \Z_{\geq 0}$ and $n \in \Z_{>0}$,
which is functorial in $V$.
If $V$ is projective, 
$c^{i, j}_{V, n}$ coincides with the map
\eqref{cetale}
under the identification \eqref{h-ch}
and
$H^i_{et, c}(V, \Z/n(j)) = H^i_{et}(V, \Z/n(j)).$
The maps $c_{X, n}^{i, j}$ are
functorial with respect to the sequences \eqref{gysin}
and the Gysin sequence in etale cohomology.
By taking the inductive limit, we also have
\[
c^{i, j}_{V, \infty} : 
 H_{2d-i}^M(V, \Q/\Z(d-j)) \to H^i_{et, c}(V, \Q/\Z(j)).
\]
The source of the map $c^{2d+r, d+r}_{V, n}$ 
can be identified with
$C_r(V)/n$ by Theorem \ref{cg-h}.

\begin{remark}
In \cite[Proposition 3.1 and p. 26 Remarks (b)]{ss2},
this homomorphism is constructed
when $k$ is a finite field.
In \cite[p. 26 Remarks (a)]{ss2},
it is pointed out that the same construction works
over any perfect field,
by using relative Poincar\'e duality
instead of its absolute version.
When $k$ is a $p$-adic field,
one can also use the absolute version
(based on $H^2_{\Gal}(k, \Z/n\Z(1)) \cong \Z/n$).
\end{remark}

\begin{proposition}\label{bydimension}
Let $V$ be a smooth variety over $F$ of dimension $d$,
and let $i, j \in \Z_{\geq 0}, n \in \Z_{>0}$.
If $i \leq j$ or $j \geq d+ \cd F$ then
$c^{i, j}_{V, n}$ is an isomorphism.
If $i=j+1$, then
$c^{i, j}_{V, n}$ is an injection.
\end{proposition}
\begin{proof}
By induction on $\dim V$ and \eqref{gysin},
this can be reduced to the case where $V$ is projective,
which is proved in Theorem \ref{bigth}.
\end{proof}

\subsection{Curves}
Let $X$ be a smooth projective irreducible curve over $F$,
and let $U$ be an open dense subscheme of $X$.
Put $Z = X \setminus U$.

\begin{lemma}\label{cv-lemma}
Let $n \in \Z_{>0} \cup \{ \infty \}$.
\begin{enumerate}
\item
We have $\ker(c^{2, 1}_{U, n})=0$.
The map $\coker(c^{2, 1}_{U, n}) \to \coker(c^{2, 1}_{X, n})$
is injective.
\item
We have $\ker(c^{3, 2}_{U, n})=0$.
If $\cd F \leq 2$,
then we have $\coker(c^{3, 2}_{U, n}) \cong \coker(c^{3, 2}_{X, n})$.
\end{enumerate}
\end{lemma}
\begin{proof}
We consider a commutative diagram with exact rows
\[
\begin{matrix}
H_1^M(X, \Z/n(0)) & \to & C_1(Z)/n & \to & C_0(U)/n & \to & C_0(X)/n &\to 0
\\
\downarrow^{c^{1,1}_{X,n}} & &  
\downarrow^{c^{1,1}_{Z,n}} & &  
\downarrow^{c^{2,1}_{U,n}} & &  
\downarrow^{c^{2,1}_{X,n}} &
\\
H^1_{et}(X, \Z/n(1)) & \to & 
H^1_{et}(Z, \Z/n(1)) & \overset{(*)}{\to} & 
H^2_{et, c}(U, \Z/n(1)) & \to & 
H^2_{et}(X, \Z/n(1)) & 
\end{matrix}
\]
The two left vertical maps are bijective
and two right vertical maps are injective
by Proposition \ref{bydimension},
and (1) follows.
The proof of (2) is similar.
%
\end{proof}

\subsection{Surfaces}
Let $X$ be a smooth projective irreducible surface over $F$.
Let $V \subset X$ be an open subvariety 
such that $\dim(X \setminus V)=0$.
Let $U \subset V$ be an open subvariety 
such that $V \setminus U$ is a (not necessary connected) smooth curve.
Given a smooth surface $U$, one can always find such $V$ and $X$.
Let $C$ be the smooth compactification of $V \setminus U$
(that is, the normalization of $X \setminus U$).

\begin{proposition}\label{keyprop}
Let $n \in \Z_{>0} \cup \{ \infty \}$.
\begin{enumerate}
\item
There is a canonical homomorphism
\[ \eta_n : \coker(c^{3, 2}_{V, n}) 
    \to \coker(c^{3, 2}_{V \setminus U, n}), 
\]
for which there are exact sequences
\begin{gather*}
\ker(\eta_n) 
\overset{(*)}{\to} 
\ker(c^{4, 2}_{U, n}) \to \ker(c^{4, 2}_{X, n}),
\\
\coker(\eta_n) 
\overset{(**)}{\to} 
\coker(c^{4, 2}_{U, n}) \to \coker(c^{4, 2}_{V, n}).
\end{gather*}
If $H^3_{et, c}(V, \Z/n(2)) \to H^3_{et, c}(V \setminus U, \Z/n(2))$
is injective, then $(*)$ is injective too.
If $c^{4, 2}_{X, n}$ is injective, then $(**)$ is injective too.
If $\cd F \leq 2$,
then we have
$\coker(c^{3, 2}_{V \setminus U, n}) \cong \coker(c^{3, 2}_{C, n}),$
$\coker(c^{3, 2}_{V, n}) \cong \coker(c^{3, 2}_{X, n})$ and
$\coker(c^{4, 2}_{V, n}) \cong \coker(c^{4, 2}_{X, n})$
(i.e., one can replace $V \setminus U$ and $V$ by $C$ and $X$
in the above sequences).
\item
Suppose $F$ is a finite extension of $\Q_p$.
Then there is an exact sequence
\begin{align*}
\coker(c^{4, 3}_{X, n}) \overset{}{\to} 
\ker(c^{5, 3}_{U, n}) \to \ker(c^{5, 3}_{X, n}),
\end{align*}
and we have
$\coker(c^{5, 3}_{U, n}) \cong \coker(c^{5, 3}_{X, n})$.
\end{enumerate}
\end{proposition}

Before we give a proof of this proposition,
we record a well-known lemma
which describes the kernel and cokernel of $c^{i, j}_{X,n}$.
We set $\mathcal{H}^{i}(\Z/n(j)) := R^i \pi_* \mu_n^{\otimes j}$
for $n \in \Z_{>0} \cup \{ \infty \}$ and $i, j \in \Z$.
(See Theorem \ref{bigth} for the definition of $\pi$.)

\begin{lemma}\label{unram}
\begin{enumerate}
\item
Let $n \in \Z_{>0} \cup \{ \infty \}.$
There is an exact sequences
\begin{align*}
0 \to &H^M_1(X, \Z/n(0)) 
\overset{c_{X, n}^{3, 2}}{\to}
H^3_{et}(X, \Z/n\Z(2)) 
\\
\to 
&H^0_{\Zar}(X, \mathcal{H}^3(\Z/n(2)))
\to 
C_0(X)/n 
\overset{c_{X, n}^{4, 2}}{\to} 
H^4_{et}(X, \Z/n(2)).
\end{align*}
\item
Let $n \in \Z_{>0} \cup \{ \infty \}.$
There is an exact sequence
\begin{align*}
0 &\to 
H_0^M(X, \Z/n(-1))
\overset{c^{4, 3}_{X, n}}{\to}
H^4_{et, c}(X, \Z/n(3)) \to
H^0_{\Zar}(X, \mathcal{H}^4(\Z/n(3)))
\\
&\to C_1(X)/n 
\overset{c^{5, 3}_{X, n}}{\to}
H^5_{et, c}(X, \Z/n(3)) \to
H^1_{\Zar}(X, \mathcal{H}^4(\Z/n(3)))
\to 0.
\end{align*}
\item
Let $n \in \Z_{>0}.$
The canonical map
$H^0_{\Zar}(X, \mathcal{H}^i(\Z/n(i-1)))
\to H^0_{\Zar}(X, \mathcal{H}^i(\Q/\Z(i-1)))[n]$
is bijective for $i=3, 4$.
\end{enumerate}
\end{lemma}

\begin{proof}
Since $X$ is projective, 
(1) and (2) follow from \eqref{h-ch}
and Bloch-Ogus theory.
(3) is a consequence of Bloch-Kato conjecture.
\end{proof}

%

\noindent
{\it Proof of Proposition \ref{keyprop} (1).}
We have a commutative diagram with exact rows:
\[
\begin{matrix}
H_1^{M}(V, \Z/n\Z(0))& \to 
C_1(V \setminus U)/n & \to 
C_0(U)/n & \to 
C_0(V)/n  \to 0
\\
\downarrow^{c_{V, n}^{3, 2}} & 
\downarrow^{c_{V \setminus U, n}^{3, 2}} & 
\downarrow^{c_{U, n}^{4, 2}} & 
\downarrow^{c_{V, n}^{4, 2}} 
\\
H^3_{et, c}(V, \Z/n\Z(2))& \to 
H^3_{et, c}(V \setminus U, \Z/n\Z(2))& \to 
H^4_{et, c}(U, \Z/n\Z(2))& \to 
H^4_{et, c}(V, \Z/n\Z(2)). 
\end{matrix}
\]
Proposition \ref{bydimension} shows 
the injectivity of $c_{V, n}^{3, 2}$ 
and $c_{V \setminus U, n}^{3, 2}$.
If $\cd F \leq 2,$ then
we have
$\coker(c_{V \setminus U, n}^{3, 2}) \cong \coker(c_{C, n}^{3, 2})$
by Lemma \ref{cv-lemma}.
Using the following lemma,
a diagram chase completes the proof.
\qed

\begin{lemma}\label{aux1}
The map
$\ker(c_{V, n}^{4, 2}) \to \ker(c_{X, n}^{4, 2})$
is injective.
If $\cd F \leq 2$, 
$\coker(c_{V, n}^{i, 2}) \to \coker(c_{X, n}^{i, 2})$
is bijective for $i=3, 4$.
\end{lemma}

\begin{proof}
This follows from
Proposition \ref{bydimension0}, \ref{bydimension} and \eqref{gysin}.
%
%
\end{proof}

\noindent
{\it Proof of Proposition \ref{keyprop} (2).}
We have a commutative diagram with exact rows:
\[
\begin{matrix}
H_0^{M}(V, \Z/n\Z(1))& \to 
C_2(V \setminus U)/n & \to 
C_1(U)/n & \to 
C_1(V)/n  \to 0
\\
\downarrow^{c_{V, n}^{4, 3}} & 
\downarrow^{c_{V \setminus U, n}^{4, 3}} & 
\downarrow^{c_{U, n}^{5, 3}} & 
\downarrow^{c_{V, n}^{5, 3}} 
\\
H^4_{et, c}(V, \Z/n\Z(3))& \to 
H^4_{et, c}(V \setminus U, \Z/n\Z(3))& \to 
H^5_{et, c}(U, \Z/n\Z(3))& \overset{}{\to} 
H^5_{et, c}(V, \Z/n\Z(3)) \to 0. 
\end{matrix}
\]
Proposition \ref{bydimension} shows 
the injectivity of the left two vertical maps.
In Proposition \ref{high2} below, 
we will show that
$c_{V \setminus U, n}^{4, 3}$ is bijective.
Using the following lemma,
a diagram chase completes the proof.
\qed

\begin{lemma}\label{aux2}
The map
$\ker(c_{V, n}^{5, 3}) \to \ker(c_{X, n}^{5, 3})$ 
is injective.
If $\cd F \leq 2$, 
then
$\coker(c_{V, n}^{i, 3}) \to \coker(c_{X, n}^{i, 3})$
is bijective for $i=3, 4$.
\end{lemma}
\begin{proof}
This follows from
Proposition \ref{bydimension0}, \ref{bydimension} and \eqref{gysin}.
%
\end{proof}

\section{Varieties over a $p$-adic field}\label{sect-var}
In this section, $k$ is a finite extension of $\Q_p$.
Let $X$ be a smooth geometrically irreducible variety over $k$
of dimension $d$.

\subsection{The Brauer-Manin pairing}\label{subsec-bm}
By Poincar\'e duality, we have a canonical isomorphism
\[H^{2d}_{et, c}(X, \hat{\Z}(d)) \cong H^2_{et}(X, \Q/\Z(1))^*.  \]
Since $\Br(X)$ is a torsion group,
the Kummer sequence implies  an exact sequence
\[ 0 \to \Pic(X) \otimes \Q/\Z 
     \to H^2_{et}(X, \Q/\Z(1)) \to \Br(X) \to 0,
\]
hence we have an injective homomorphism
$\Br(X)^* \to H^{2d}_{et, c}(X, \hat{\Z}(d))$.
We consider a diagram
\[
\begin{matrix}
Z_0(X) & \overset{\tilde{\psi}_X}{\to} & \Br(X)^*
\\
\downarrow^{\text{surj.}}
& &
\downarrow_{\text{inj.}}
\\
C_0(X) & \to & H^{2d}_{et, c}(X, \hat{\Z}(d)),
\end{matrix}
\]
where the lower horizontal map is the composition of
\[
C_0(X) 
\to \underset{\leftarrow n}{\lim} ~C_0(X)/n
\overset{(c^{2d, d}_{X, n})_n}{\to} H^{2d}_{et, c}(X, \hat{\Z}(d)).
\]
This diagram is seen to be commutative 
by reducing to the case $\dim X=0$.
This shows 
the first part of Proposition \ref{well-def}
(the well-definedness of $\psi_X$).
The same argument shows the following simple lemma.

\begin{lemma}\label{candpsi}
For any $n \in \Z_{>0} \cup \{ \infty \}$, we have an isomorphism
$\ker(\psi_X/n) \cong \ker(c^{2d, d}_{X, n})$
and an exact sequence
\[ 0 \to \coker(\psi_X/n) \to \coker(c^{2d, d}_{X, n})
 \to (\Pic(X)/n)^* \to 0.
\]
\end{lemma}

\subsection{Class field theory}\label{subsec-cft}
By Poincar\'e duality, we have a canonical isomorphism
\[H^{2d+1}_{et, c}(X, \hat{\Z}(d+1)) 
\cong H^1_{et}(X, \Q/\Z)^*
\cong \pi_1^{ab}(X).  
\]
We consider a diagram
\[
\begin{matrix}
Z_0^1(X) & \overset{\tilde{\rho}_X}{\to} & \pi_1^{ab}(X)
\\
\downarrow^{\text{surj.}}
& &
\downarrow_{\cong}
\\
C_1(X) & \to & H^{2d+1}_{et, c}(X, \hat{\Z}(d+1)).
\end{matrix}
\]
where the lower horizontal map is the composition of
\[
C_1(X) 
\to \underset{\leftarrow n}{\lim} C_1(X)/n
\overset{(c^{2d+1, d+1}_{X, n})_n}{\to} H^{2d+1}_{et, c}(X, \hat{\Z}(d+1)).
\]
This diagram is seen to be commutative 
by reducing to the case $\dim X=0$.
This shows 
the second part of Proposition \ref{well-def}
(the well-definedness of $\rho_X$).
The same argument shows the following simple lemma.

\begin{lemma}\label{candrho}
For any $n \in \Z_{>0} \cup \{ \infty \}$, we have isomorphisms
$\ker(\rho_X/n) \cong \ker(c^{2d+1, d+1}_{X, n})$ and
$\coker(\rho_X/n) \cong \coker(c^{2d+1, d+1}_{X, n})$.
\end{lemma}

\subsection{Higher degree}
We consider the groups $C_r(X)$ when $r \geq 2$.
Proposition \ref{high2} below shows that
$C_r(X)$ is uniquely divisible if $r \geq 3$,
and that $C_2(X) \to C_2(\Spec k) = K_2(k)$
is a surjective map with uniquely divisible kernel.
Note that 
$K_2(k)$ is the direct sum of a uniquely divisible group
and a finite group isomorphic to $\mu(k)$
(see \cite{Merkurjev}).

\begin{proposition}\label{high2}
Let $X$ be a smooth geometrically irreducible variety over $k$.
Suppose that there exists a smooth projective variety $Y$
that contains $X$ as an open dense subvariety.
Let $i, j \in \Z$ and suppose $j \leq -2$.
\begin{enumerate}
\item
If $i < -2$,
then $H^M_i(X, \Z(j))$ is uniquely divisible.
\item
The map $H^M_{-2}(X, \Z(j)) \to H^M_{-2}(\Spec k, \Z(j))$
induced by the structure morphism
is surjective with uniquely divisible kernel.
Consequently,
\[ c^{2d+2, d+2}_{X, n} : C_2(X)/n \to 
   H^{2d+2}_{et, c}(X, \Z/n(d+2))
\]
is bijective for any $n \in \Z_{>0}$.
\end{enumerate}
\end{proposition}

\begin{proof}
Since $j \leq -2$, Proposition \ref{bydimension}
shows that
\begin{equation}\label{tempuse}
c^{2d-i, d-j}_{X, n} : 
H_{i}^M(X, \Z/n(j)) \to H^{2d-i}_{et, c}(X, \Z/n(d-j))
\end{equation}
is an isomorphism for any $n \in \Z_{>0} \cup \{ \infty \}$.
If $i < -2$, then we have
$H^{2d-i}_{et, c}(X, \Z/n\Z(d-j)) = 0$ since $\cd k = 2$.
We also have $H^{2d+2}_{et, c}(X, \Q/\Z(d-j)) = 0$,
because this group is dual to $H^0_{\Gal}(k, \hat{\Z}(j+1))^* = 0$
under Poincar\'e duality.
Now (1) follows from \eqref{univcoeff}.

(2) follows from the four claims (a)-(d) below.
In what follows we write
$H_*^M(k, -)$ for $H_*^M(\Spec k, -)$.
Put $H_{-2}^M(X, \Z(j))_0 
:= \ker[H_{-2}^M(X, \Z(j)) \to H_{-2}^M(k, \Z(j))]$.

\begin{itemize}
\item[(a)]
$H_{-2}^M(X, \Z(j)) \to H_{-2}^M(k, \Z(j))$
is surjective.
\item[(b)]
$H_{-2}^M(X, \Z(j))/n \to H_{-2}^M(k, \Z(j))/n$
is bijective for any $n \in \Z_{>0}$.
\item[(c)]
$H_{-2}^M(X, \Z(j))[n] \to H_{-2}^M(k, \Z(j))[n]$
is surjective for any $n \in \Z_{>0}$.
\item[(d)]
$H_{-2}^M(X, \Z(j))_0$ is torsion free.
\end{itemize}
Indeed, by (a) we have
an exact sequence for any $n \in \Z_{>0}$
\begin{align*}
 0 \to &H_{-2}^M(X, \Z(j))_0[n] \to
 H_{-2}^M(X, \Z(j))[n] \to H_{-2}^M(k, \Z(j))[n]
\\
 & \to H_{-2}^M(X, \Z(j))_0/n \to
 H_{-2}^M(X, \Z(j))/n \to H_{-2}^M(k, \Z(j))/n \to 0,
%
%
%
%
\end{align*}
which shows (2) in view of (b)-(d).

We prove claims (a)-(d).
Since we have shown 
$H_{-3}^M(X, \Z(j))_{\Tor} = 0$,
\eqref{univcoeff} and \eqref{tempuse} yield a commutative diagram
\[
\begin{matrix}
H_{-2}^M(X, \Z(j))/n & \cong & H_{et, c}^{2d+2}(X, \Z/n\Z(d-j)) \\
\downarrow & & \downarrow_{\cong} \\
H_{-2}^M(k, \Z(j))/n & \cong & H_{\Gal}^{2}(k, \Z/n\Z(-j)),
\end{matrix}
\]
which shows (b). 
Taking a closed point $x \in X$,
the cokernel of 
$H_{-2}^M(X, \Z(j)) \to H_{-2}^M(k, \Z(j))$
is seen to be annihilated by $[k(x):k]$.
Using (b) with $n=[k(x):k]$,
we get (a).
By \eqref{univcoeff}
we have a commutative diagram with exact rows:
\[
\begin{matrix}
&H^{2d+1}_{et, c}(X, \Z/n\Z(d-j)) & \to 
&H_{-2}^M(X, \Z(j))[n] &
 \to 0
\\
&{}_{(*)} \downarrow&
&\downarrow&
\\
&H^{1}_{\Gal}(k, \Z/n\Z(-j)) & \to 
&H_{-2}^M(k, \Z(j))[n] &
 \to 0.
\end{matrix}
\]
The map $(*)$ is induced by the structure 
morphism $\alpha : X \to \Spec k$.
Thus it is dual to 
$\alpha^* : H^1_{\Gal}(k, \Z/n\Z(1+j)) \to H^1_{et}(X, \Z/n\Z(1+j))$,
which is injective by 
Leray spectral sequence.
This shows that $(*)$ is surjective,
and (c) follows.
We consider a commutative diagram
with exact rows and columns
\[
\begin{matrix}
&&
&&
&0&
\\
&&
&&
&\downarrow&
\\
& &
&H^{2}_{\Gal}(k, H^{2d-1}_{et, c}(\bar{X}, \Q/\Z(d-j))) & 
&H_{-2}^M(X, \Z(j))_{0, \Tor} &
\\
&&
&\downarrow&
&\downarrow&
\\
 0 \to 
&H_{-1}^M(X, \Z(j))\otimes \Q/\Z& \to
&H^{2d+1}_{et, c}(X, \Q/\Z(d-j)) & \to 
&H_{-2}^M(X, \Z(j))_{\Tor} &
 \to 0
\\
&\downarrow&
&\downarrow&
&\downarrow&
\\
 0 \to 
&H_{-1}^M(k, \Z(j))\otimes \Q/\Z& \to
&H^{1}_{\Gal}(k, \Q/\Z(d-j)) & \to 
&H_{-2}^M(k, \Z(j))_{\Tor} &
 \to 0.
\end{matrix}
\]
Note that the lower left vertical map is surjective
since the cokernel of 
$H_{-1}^M(X, \Z(j)) \to H_{-1}^M(k, \Z(j))$ is torsion
(annihilated by $[k(x):k]$ for any closed point $x \in X$).
Now the following lemma completes the proof of (d).
\end{proof}

\begin{lemma}\label{etalecoh2}
Let $X$ be a variety over $k$ 
satisfying the same assumption as in
Proposition \ref{high2}.
If $j \leq -2$, then
$H^2_{\Gal}(k, H^{2d-1}_{et, c}(\bar{X}, \Q/\Z(d-j))=0$.
\end{lemma}
\begin{proof}
By Poincar\'e duality,
it suffices to show
$H^0_{\Gal}(k, H^1_{et}(\bar{X}, \Z_l(j+1)))=0$
for all prime number $l$.
There is an exact sequence
\[ 0 \to H^{1}_{et}(\bar{Y}, \Z_l(j+1))
  \to H^{1}_{et}(\bar{X}, \Z_l(j+1))
  \to H^{2}_{et, \bar{Y} \setminus {X}}(\bar{Y}, \Z_l(j+1)).
\]
By excision and purity, we have
We have
$H^0_{\Gal}(k, H^{2}_{et, \bar{Y} \setminus {X}}(\bar{Y}, \Z_l(j+1)))=0$.
Thus we are reduced to showing
$H^0_{\Gal}(k, H^{1}_{et}(\bar{Y}, \Z_l(j+1)))=0$,
but this is a result of Jannsen
\cite[Theorem 4.2 and 5.3 for the case $l \not= p$
and $l=p$ respectively]{jannsen}.
\end{proof}

\subsection{Auxiliary lemmas}
For future use, we record a few simple lemmas.
The following lemma is often used
concurrently with the fact that
the groups $\pi_1^{ab}(X)$ and $\Br(X)^*$
have no non-trivial divisible elements
for any smooth variety $X$.
(Cf. \S \ref{convention}.)

\begin{lemma}\label{easylemma}
Let $f: A \to B$ be a homomorphism of abelian groups.
\begin{enumerate}
\item
Suppose $B_{\divi}=0$.
Then there is an exact sequence
\[ 0 \to A_{\divi} \to \ker(f) \to 
\underset{\leftarrow}{\lim} \ker(f/n).
\]
\item
Suppose that there is an $N \in \Z_{>0}$ such that
the canonical map $\ker(f/nN) \to \ker(f \otimes \Q/\Z)$ 
is bijective for all $n \in \Z_{>0}$.
Then we have
$\underset{\leftarrow}{\lim} \ker(f/n) = 0$.
(Hence $\ker(f)=A_{\divi}$ if $B_{\divi}=0$.)
\end{enumerate}
\end{lemma}
\begin{proof}
(1) is a direct consequence of the definition.
We prove (2).
By assumption,
the canonical map 
$\ker(f/mN) \to \ker(f/mnN)$ 
is bijective for all $n, m \in \Z_{>0}$,
which implies that the canonical map
$\ker(f/mnN) \to \ker(f/nN)$ is the zero-map.
Thus we have
$\underset{\leftarrow}{\lim} \ker(f/n) = 0$.
\end{proof}

%

\begin{lemma}\label{divdiv}
Let $X$ be a smooth irreducible variety over $k$ of dimension $d$,
and let $i, j \in \Z$.
If $i+2 \geq j+d$ or $j \geq d+2$, then
we have $H_i^M(X, \Z(j))_{\divi}=H_i^M(X, \Z(j))_{\Div}$.
In particular, 
we have $C_r(X)_{\divi}=C_r(X)_{\Div}$ 
for any $r \in \Z_{\geq 0}$ when $d \leq 2$.
If $d=1$, then $C_0(X)_{\Div}=0$.
\end{lemma}

\begin{proof}
It suffices to show that
$H_i^M(X, \Z(j))[n]$ is finite for any $n \in \Z_{>0}$.
Proposition \ref{bydimension} and \eqref{univcoeff}
reduce this to the finiteness of
$H^{2d-i-1}_{et, c}(X, \Z/n(d-j))$,
which is well-known.
For the last statement, we recall that
$C_0(X)$ is isomorphic to the relative Picard group
(cf. Remark \ref{cg-rem}),
which has no non-trivial divisible subgroup
by a theorem of Mattuck \cite{mat}.
\end{proof}

\section{Curves over a local field}\label{sect-cv}
In this section, $k$ is a finite extension of $\Q_p$.
Let $X$ be a smooth projective irreducible curve over $k$
and let $U$ be an open dense subscheme of $X$.
Put $Z = X \setminus U$.

\subsection{The Brauer-Manin pairing}
The following theorem was proved by
Scheiderer and van Hamel \cite[(3.5)]{svh}.
When $U=X$, this theorem is due to Lichtenbaum \cite{lichtenbaum}.

\begin{theorem}[Scheiderer/van Hamel \cite{svh}]\label{thm-svh}
The homomorphism $\psi_U : C_0(U) \to \Br(U)^*$
is an injection with dense image.
The induced map
$\psi_U/n : C_0(U)/n \to \Br(U)^*/n$ 
is an isomorphism
for all $n \in \Z_{>0}$.
\end{theorem}

We deduce Theorem \ref{thm-svh} assuming Lichtenbaum's result.
We may assume $X$ is geometrically irreducible over $k$.
Let $n \in \Z_{>0}$.
Lemmas \ref{cv-lemma} and \ref{candpsi} 
show the injectivity of $\psi_U/n$.
By Lemma \ref{candpsi},
we have a commutative diagram
\[
\begin{matrix}
0 \to \coker(\psi_U/n) \to 
&\coker(c^{2, 1}_{U, n}) & \to & (\Pic(U)/n)^*
\to 0
\\
&\downarrow^{(*)} & &  \downarrow 
\\
&\coker(c^{2, 1}_{X, n}) & \cong & (\Pic(X)/n)^*.
\end{matrix}
\]
The bottom horizontal map is bijective 
because $\coker(\psi_X/n)=0$ by Lichtenbaum's result.
By Lemma \ref{cv-lemma}, $(*)$ is injective,
thus $\coker(\psi_U/n)=0$.
The rest of assertion follows from
Lemmas \ref{easylemma} and \ref{divdiv}.
\qed

\subsection{Class field theory}
We recall a result for projective curves
due to Bloch and Saito.

\begin{theorem}[\cite{bloch2} \cite{saito1}]\label{shuji}
\begin{enumerate}
\item
The kernel of the homomorphism 
$\rho_X : C_1(X) (= SK_1(X)) \to \pi_1^{ab}(X)$
coincides with $C_1(X)_{\Div}$.
Moreover, $\rho_X/n$ is injective 
for all $n \in \Z_{>0}$. 
\item
The group 
$D(X) := \pi_1^{ab}(X)/\overline{\Image(\rho_X)}$
is isomorphic to 
$\hat{\Z}^{\oplus r(X)}$ for some $r=r(X) \in \Z_{\geq 0}$.
We have $\coker(\rho_X/n) \cong D(X)/n$ 
for any $n \in \Z_{>0}$.
\end{enumerate}
\end{theorem}

\begin{remark}\label{rem-genus-zero}
If $X$ has potentially good reduction, then $r(X)=0$.
In general, we have an inequality 
$r(X) \leq g(X)$, where $g(X)$ the genus of $X$
\cite[(6.2)]{saito1}.
In particular, 
We have $D(X)=0$ when $g(X)=0$.
\end{remark}

The following generalization follows from 
Lemmas \ref{cv-lemma}, \ref{candrho}, \ref{easylemma} and \ref{divdiv}.

\begin{theorem}\label{thm-hira}
\begin{enumerate}
\item
The kernel of the homomorphism 
$\rho_U : C_1(U) \to \pi_1^{ab}(U)$
coincides with $C_1(U)_{\Div}$.
Moreover, $\rho_U/n$ is injective 
for all $n \in \Z_{>0}$. 
\item
Set $D(U) := \pi_1^{ab}(U)/\overline{\Image(\rho_U)}$. 
The canonical maps
$D(U) \to D(X)$
and
$\coker(\rho_U/n) \to \coker(\rho_X/n) (\cong D(X)/n)$
are bijective for any $n \in \Z_{>0}$.
\end{enumerate}
\end{theorem}

\begin{remark}\label{rem-hira}
Following Hiranouchi \cite{hira}, we define
\[  C_1^w(U) := \coker[
K_2(k(X)) \to (\bigoplus_{x \in U_{(0)}} k(x)^*) \oplus
(\bigoplus_{z \in Z} K_2(k(X)_z)],
\]
where $k(X)_z$ is the completion of $k(X)$ at $z$.
He constructed the reciprocity map
$\rho_U': C_1^w(U) \to \pi_1^{ab}(U)$
and showed that the kernel of $\rho_U'$ coincides with $C_1^w(U)_{\Div}$,
and that $\pi_1^{ab}(U)/\overline{\Image(\rho_U')} 
\cong D(X)$.
It is easy to see the following:
\begin{itemize}
\item
The natural projection $\pi :C_1^w(U) \to C_1(U)$
fits into an exact sequence
\[ \oplus_{z \in Z} UK_2 k(X)_z \to C_1^w(U) 
\overset{\pi}{\to} C_1(U) \to 0. \]
Here we put
$UK_2 k(X)_z = \ker[K_2 k(X)_z \to k(z)^* \oplus K_2 k(z)]$
where the map is defined by
$a \mapsto (\pd_z(a), \pd_z(\{ \pi_z \} \cup a)$
for some uniformizer $\pi_z \in k(X)_z$ at $z$
(and $\pd_z$ is the tame symbol).
Note that $UK_2 k(X)_z$ is uniquely divisible.
\item
We have $\rho_U' = \rho_U \circ \pi$.
\end{itemize}
Hence Hiranouchi's result can be recovered by Theorem \ref{thm-hira}.
\end{remark}

\section{Surfaces over a local field}\label{sect-surface}
We keep assuming $k$ is a finite extension of $\Q_p$.
Let $X$ be a smooth projective geometrically connected surface over $k$.
Let $U \subset V \subset X$ be open dense subsets
such that $\dim(X \setminus V)=0$ and
such that $V \setminus U$ is a smooth curve.
Let $C$ be the normalization of $X \setminus U$,
and set $D(C) = \oplus_i D(C_i)$, 
where $C_i$ are the irreducible components of $C$,
and $D(C_i)$ is the group defined in Theorem \ref{shuji}.

\subsection{$\mathbb{P}^2$ minus curves}
As the first example, we show the following.
\begin{proposition}\label{ptwo}
Suppose $X \cong \mathbb{P}^2$.
Let $d$ be the greatest common divisor of
the degrees (as a subvariety of $\mathbb{P}^2$)
of the irreducible components of $X \setminus U$.
Then, there is an exact sequence
\[ 0 \to C_0(U)/nd \overset{\psi_U/nd}{\to} \Br(U)^*/nd  
   \to D(C)/nd \to 0
\]
for any $n \in \Z_{>0}$.
We have
$\ker(\psi_U) = C_0(U)_{\Div}$ and
$\Br(U)^*/\overline{\Image(\psi_U)} \cong D(C)$.
\end{proposition}

For the map $\rho_U : C_1(U) \to \pi_1^{ab}(U)$,
see Theorem \ref{rational1} below.

\begin{proof}
Let $n \in \Z_{>0}$.
Note that
$\psi_X/n : C_0(X)/n \to \Br(X)^*/n$ is bijective,
and that
$H^0_{\Zar}(X, \mathcal{H}^i(\Q/\Z(i-1)))=0$ for $i=3, 4$.
By Lemmas \ref{unram}, \ref{candpsi},
and Proposition \ref{keyprop},
we have $\ker(\psi_U/n) = 0$.
By Lemmas \ref{easylemma} and \ref{divdiv}, 
we get $\ker(\psi_U) = C_0(U)_{\Div}$.

Note that
$\Z = \Pic(X) \to \Pic(U)$ is a surjection with kernel $d\Z$.
By Lemma \ref{candpsi},
we get a commutative diagram with exact rows
\[ 
\begin{matrix}
0 \to  
& \coker(\psi_U/nd)&  \to 
& \coker(c^{4, 2}_{U, nd})&  \to  
& \Z/d\Z  & \to 0
\\
&  &
& \downarrow &
& \downarrow^{\text{inj.}} & &
\\
0 = 
& \coker(\psi_X/nd)&  \to 
& \coker(c^{4, 2}_{X, nd})&  \cong
& \Z/nd. &
\end{matrix}
\]
By Proposition \ref{keyprop} and Lemma \ref{unram},
we get an exact sequence
\[ 0  \to \coker(c_{C,nd}^{3, 2}) \to 
\coker(c_{U, nd}^{4, 2}) \to \coker(c_{X, nd}^{4, 2}).
\] 
It follows that
$\coker(\psi_U/nd) \cong \coker(c_{C,nd}^{3, 2})$,
which is isomorphic to $D(C)/nd$
by Lemma \ref{candrho} and Theorem \ref{shuji}.
We are done.
\end{proof}

\subsection{Rational surface}
Suppose that $\bar{X} := X \times_{\Spec k} \Spec \bar{k}$ 
is a {\it rational} surface over $\bar{k}$
(that is, $\bar{X}$ is birational to $\mathbb{P}^2_{\bar{k}}$),
and that $X(k) \not= \phi$.
In this situation, 
the Chow group $C_0(X)=CH_0(X)$
and related maps such as $\psi_X$ and $c^{4, 2,}_{X, n}$
have been studied by several authors
(see, for example, 
\cite{manin2, bloch, ctsansuc, ct, ct-survey, ct-survey2}).
We will briefly recall a few results among them.

We set $A_0(X) := \ker[C_0(X) = CH_0(X) \overset{\deg}{\to} \Z]$.
Let $S = \Hom(NS(\bar{X}), \bar{k}^*)$ 
be the N\'eron-Severi torus of $X$.
Recall that we have the Bloch map \cite{bloch}, \cite{ctsansuc}
\[  \phi_X : A_0(X) \to H^1_{\Gal}(k, S). \]
It is known \cite{ct} that $\phi_X$ is injective.
Note that $H^1_{\Gal}(k, S)$ is finite 
(hence so is $A_0(X)$).
We will also need the following result.

\begin{theorem}[Saito \cite{saito}, see also \cite{ct-survey}]\label{rat}
There is an $N=N_X \in \Z_{>0}$ 
such that 
$c_{X, nN}^{4, 2}$ is injective
for all $n \in \Z_{>0}$.
Moreover, we have $\ker(\psi_X)=0$.
\end{theorem}

Now we state our main result.

\begin{theorem}\label{rational}
We suppose the following the condition:
\[ (*) 
\qquad
\text{the irreducible components of }~
\bar{X} \setminus \bar{U}
~\text{generate}~
NS(\bar{X}).
\]
Then we have the following.
\begin{enumerate}
\item
$\ker(\psi_U) = C_0(U)_{\Div}$.
\item
$F_0(U) := \ker(\psi_U \otimes \Q/\Z)$ is finite.
There is an $N \in \Z_{>0}$ such that for all $n \in \Z_{>0}$
the canonical map
$\ker(\psi_U/nN) \to F_0(U)$ is bijective.
\item
There is an injection
$F_0(U) \hookrightarrow \coker(\phi_X)$,
which is bijective if $D(C)=0$.
\end{enumerate}
\end{theorem}

In Example \ref{chatelet}, \ref{cubic} below,
we shall provide concrete examples 
for which $F_0(U) \not= 0$.
Theorem \ref{rational0} (1)
is a consequence of this theorem and examples.
Note that Parimala and Suresh \cite{ps} have constructed 
smooth projective surfaces $X$ and $X'$
such that
$\ker(\psi_X) \not= C_0(X)_{\Div}$ 
and
$\ker(\psi_{X'}/2) \not= 0$.
It seems difficult to control the cokernel of $\psi_U$, 
because, as Proposition \ref{ptwo} shows,
it can easily become bigger and bigger.
If we fix $X$ and vary $U$,
the above theorem immediately implies the following.

\begin{proposition}
There exists an open dense subvariety $U_0 \subset X$ such that, 
for any open dense subvariety $U \subset U_0$, the natural map
$F_0(U) \to F_0(U_0)$ is an isomorphism.
\end{proposition}

\begin{proof}
We take $U_* \subset X$ that satisfies $(*)$.
By Theorem \ref{rational} (3),
the map $F_0(U) \to F_0(U_*)$ is injective
for any $U \subset U_*$.
Since $F_0(U_*)$ is finite,
there exists $U_0 \subset U_*$ such that
\[ F_0(U_0) = \bigcap_{U \subset U_*} F_0(U). \]
The proposition holds with this $U_0$.
\end{proof}

As for the map $\rho_U$,
we prove the following result.
Theorem \ref{rational0} (2) follows from this.

\begin{theorem}\label{rational1}
The map 
$\rho_U/n : C_1(U)/n \to \pi_1^{ab}(U)/n$
is an isomorphism
for any $n \in \Z_{>0}$.
We have
$\ker(\rho_U) = C_1(U)_{\Div}$ and
$\pi_1^{ab}(U) = \overline{\Image(\rho_U)}$.
\end{theorem}

Note that Sato \cite{sato2} has constructed 
smooth projective $K3$ surfaces $X, X'$
such that
$\ker(\rho_X) \not= C_1(X)_{\Div}$ 
and
$\ker(\rho_{X'}/2n) \not= 0$ for all $n \in \Z_{>0}$.

Before we start the proof of 
Theorems \ref{rational} and \ref{rational1},
we introduce a few lemmas.


\begin{lemma}\label{unram-galcoh}
We have a canonical isomorphism
\[ H^0_{\Zar}(X, \mathcal{H}^3(\Q/\Z(2))) 
\cong \coker[\phi_X : A_0(X) \to H^1_{\Gal}(k, S)].
\]
In particular, $H^0_{\Zar}(X, \mathcal{H}^3(\Q/\Z(2)))$ is finite.
\end{lemma}
\begin{proof}
By Kummer sequence,
we get an exact sequence for any $n \in \Z_{>0}$
\[ 0 \to S(k)/n \to H^1_{\Gal}(k, S[n]) \to H^1_{\Gal}(k, S)[n] \to 0. \]
Using the facts that $H^1_{\Gal}(k, S)$ is a torsion group,
and that Hochschild-Serre spectral sequence
induces an isomorphism
$H^3_{et}(X, \Q/\Z(2)) \cong H^1_{\Gal}(k, S_{\Tor})$,
we get an exact sequence which fits into the
lower row in the commutative diagram
(note that $CH_0(X)_{\Tor} = A_0(X)$)
\[
\begin{matrix}
0 \to 
& H^3_M(X, \Z(2)) \otimes \Q/\Z&  \to 
& H^3_M(X, \Q/\Z(2)) &  \to 
& CH_0(X)_{\Tor} &  \to 0
\\
& \downarrow^{\cong} &
& \downarrow^{c^{3, 2}_{X, \infty}} &
& \downarrow^{\phi_X} &
\\
0 \to 
& S(k) \otimes \Q/\Z &  \to 
& H^3_{et}(X, \Q/\Z(2)) &  \to 
& H^1_{\Gal}(k, S) &  \to 0.
\end{matrix}
\]
Here the upper row is exact by \eqref{univcoeff}.
The left vertical isomorphism is
given by \cite[Theorem C]{ct}.
(Here we used the assumption $X(k) \not= \phi)$.
Now the assertion follows from 
Theorem \ref{rat} and Lemma \ref{unram}.
\end{proof}

\begin{lemma}
If $(*)$ is satisfied, then
$H^3_{et, c}(V, \Q/\Z(2)) \to H^3_{et, c}(V \setminus U, \Q/\Z(2))$
is injective.
\end{lemma}
\begin{proof}
It suffices to show that
$H^3_{et, c}(U, \Q/\Z(2)) \to H^3_{et, c}(V, \Q/\Z(2))$
is the zero map.
By Poincar\'e duality,
this amounts to showing that
$H^3_{et}(V, \hat{\Z}(1)) \to H^3_{et}(U, \hat{\Z}(1))$
is the zero-map.
By Gysin sequence
$H^3_{et}(X, \hat{\Z}(1)) \to H^3_{et}(V, \hat{\Z}(1)) 
\to H^0_{\et}(X \setminus V, \hat{\Z}(-1)) = 0,$
it suffices to show that
$H^3_{et}(X, \hat{\Z}(1)) \to H^3_{et}(U, \hat{\Z}(1))$
is the zero-map.
By Hochschild-Serre spectral sequence,
we have
$H^3_{et}(X, \hat{\Z}(1)) 
= H^1_{\Gal}(k, H^2_{et}(\bar{X}, \hat{\Z}(1)))$,
and we are reduced to showing that
$NS(\bar{X}) \otimes \hat{\Z} 
= H^2_{et}(\bar{X}, \hat{\Z}(1)) \to H^2_{et}(\bar{U}, \hat{\Z}(1))$
is the zero map,
but this follows from the assumption $(*)$.
\end{proof}

\noindent
{\it Proof of Theorem \ref{rational}.}
Let $N_1 = N_X \in \Z_{>0}$ be the natural number
appearing in Theorem \ref{rat},
and $N_2$ the order of
$H^0_{\Zar}(X, \mathcal{H}^3(\Q/\Z(2)))$.
Put $N = N_1 N_2$ and take any $n \in \Z_{>0}$.
By Theorem \ref{rat} and Lemma \ref{unram}, we have
$\coker(c^{3, 2}_{X, nN}) \cong H^0_{\Zar}(X, \mathcal{H}^3(\Z/nN(2)))$.
Hence we have a commutative diagram
which contains the map $\eta_n$ appearing in Proposition \ref{keyprop}:
\[
\begin{matrix}
H^0_{\Zar}(X, \mathcal{H}^3(\Z/nN(2))) &  \overset{\eta_{nN}}{\to} &
\coker(c_{C, nN}^{3, 2}) & \cong & 
(\Z/nN)^{\oplus r(C)}
\\
\Vert &   &
\downarrow &   & 
\cap
\\
H^0_{\Zar}(X, \mathcal{H}^3(\Q/\Z(2))) &  \overset{\eta_{\infty}}{\to} &
\coker(c_{C, \infty}^{3, 2}) & \cong & 
(\Q/\Z)^{\oplus r(C)},
\end{matrix}
\]
where the two right horizontal isomorphisms are given by
Theorem \ref{shuji}.
We get $\ker(\eta_{nN}) \cong \ker(\eta_{\infty})$.
Then by Theorem \ref{rat} and Proposition \ref{keyprop},
we have a commutative diagram,
\[ 
\begin{matrix}
\ker(\eta_{nN})  & \twoheadrightarrow &
\ker(\psi_U/nN)
\\
\downarrow^{\cong}  & & \downarrow
\\
\ker(\eta_{\infty})  & \cong  &
\ker(\psi_U \otimes \Q/\Z) =F_0(U),
\end{matrix}
\]
in which the top horizontal map is surjective.
Hence all the maps in this diagram are bijective.
This proves (2).
If $D(C)=0$, then we have
$\ker(\eta_{\infty}) = H^0_{\Zar}(X, \mathcal{H}^3(\Q/\Z(2)))$,
and Lemma \ref{unram-galcoh} proves (3).
(1) follows from Lemma \ref{easylemma}.
\qed

\noindent
{\it Proof of Theorem \ref{rational1}.}
By Proposition \ref{keyprop} and Lemma \ref{candpsi},
it suffices to show the bijectivity of
$c^{i, 3}_{X, n}$ for $i=4, 5$ and $n \in \Z_{>0}$.
Probably this is well-known to the specialists,
but the author could not find a suitable reference.
For the completeness sake,
we include a proof here.

Let $n \in \Z_{>0} \cup \{ \infty \}$.
The injectivity of 
$c^{4, 3}_{X, n}$ follows from Proposition \ref{bydimension}.
We consider the commutative diagram with exact rows
\begin{equation}\label{ratcfteq}
\begin{matrix}
0 \to 
H_0^M(X, \Z(-1))/n& \to 
&H_0^M(X, \Z/n(-1))& \to 
&C_1(X)[n] \to 0
\\
 \downarrow & 
& \downarrow^{c^{4, 3}_{X, n}} & 
& \downarrow 
\\
0 \to 
H^2_{\Gal}(k, NS(\bar{X}) \otimes \Z/n(2))& \to 
&H_{et}^4(X, \Z/n(3))& \to 
&H^0_{\Gal}(k, \Z/n(1)) \to 0.
\end{matrix}
\end{equation}
Here the upper and lower rows
are given by \eqref{univcoeff} and by
Hochschild-Serre spectral sequence respectively.
The right vertical map is given by
\[ C_1(X)[n] \to C_1(\Spec k)[n] = k^*[n] = H^0_{\Gal}(k, \Z/n(1)). \]
The left vertical map is defined by the commutativity
of this diagram.
The existence of a $k$-rational point of $X$
shows that both rows are split exact.
It also shows the surjectivity of the right vertical map.
We show the surjectivity of the left vertical map.
Since the corestriction map
\[ H^2_{\Gal}(k', NS(\bar{X}) \otimes \Z/n(2)) \to 
   H^2_{\Gal}(k, NS(\bar{X}) \otimes \Z/n(2))
\]
is surjective for any finite extension $k'/k$,
it suffices to show this surjectivity
when $X$ is a {\it split} rational surface
(that is, $k(X)$ is purely transcendental over $k$).
This follows from Lemma \ref{unram}
because $H^0_{\Zar}(X, \mathcal{H}^4(\Z/n(3)))=0$
when $X$ is a split rational surface.
(More explicitly, one can directly show the surjectivity
of the composite map
\[
\Pic(X) \otimes K_2 k \overset{\text{prod.}}{\to} H^4_M(X, \Z(3)) \cong 
H^M_0(X, \Z(-1)) \to 
H^2_{\Gal}(k, NS(\bar{X}) \otimes \Z/n(2)),
\]
where the first map is given by the product structure.)

Next, we consider $c^{5, 3}_{X, n}$.
By Hochschild-Serre spectral sequence,
we have an isomorphism
\[ 
H^5_{et}(X, \Z/n(3)) \cong H^1_{\Gal}(k, \Z/n(1)) \cong k^*/n.
\]
It suffices to show that
the structure morphism induces
an isomorphism $C_1(X) \to C_1(\Spec k) = k^*$.
By the existence of a $k$-rational point of $X$,
one knows that this map is split surjective.
We set $V(X) := \ker[C_1(X) \to C_1(\Spec k)=k^*]$.
Since $V(X')=0$ if $X'$ is a split rational surface,
the norm argument shows that $V(X)$ is a torsion group.
It suffices to show $V(X)_{\Tor}=0$.
This follows from a commutative diagram with exact rows
\[
\begin{matrix}
0 \to V(X)_{\Tor} \to & C_1(X)_{\Tor}& \overset{(!)}{\to} & \mu(k) \to 0
\\
& \uparrow^{\text{surj.}} & & \Vert
\\
& H_0^{M}(X, \Q/\Z(-1))& \overset{(!)}{\to} & H_0^M(\Spec k, \Q/\Z(1))
\\
& \downarrow^{c^{4, 3}_{X, \infty}} & & \Vert
\\
& H^4_{et}(X, \Q/\Z(3))& \cong & H^0_{\Gal}(k, \Q/\Z(1)).
\end{matrix}
\]
Here the maps $(!)$ are split surjective,
and $c^{4, 3}_{X, \infty}$ is injective by Proposition \ref{bydimension}.
The lower horizontal map is bijective
because the first term in the bottom row of \eqref{ratcfteq}
vanishes when $n=\infty$.
The middle upper vertical map is given by \eqref{univcoeff}.
\qed

\begin{example}\label{chatelet}
Let $a, b \in k^*$.
Suppose $p \not= 2$, $a \not= b$ and $\ord_k(a)=\ord_k(b)=r$.
We also take $d \in k^*$ such that $k(\sqrt{d})/k$
is a non-trivial unramified extension.
Let $X$ be a smooth projective surface with 
function field 
$K := k(x, y)[z]/(x^2 - dy^2 - z(z-a)(z-b))$
(a so-called {\it Ch\^atelet surface}).
Let $C$ be a divisor on $X$, 
and put $U=X \setminus C$.
We assume that the irreducible components of 
$\bar{C}$ generate $NS(\bar{X})$,
and that $D(C)=0$.
(When $X$ is minimal
with respect to the fibration corresponding 
to the field extension $K/k(z)$,
we can take $C$ to be the union of
the four singular fibers and the `$\infty$-section',
so that each component has genus zero. 
(Cf. Remark \ref{rem-genus-zero}.)
See \cite{ctsansuc2} for details.)
We claim the following.

\begin{enumerate}
\item
If $r$ is even and $\ord_k(a-b)=r$, then
$F_0(U) \cong \Z/2 \oplus \Z/2$.
\item
If $r$ is even and $\ord_k(a-b)>r$, then
$F_0(U) \cong \Z/2$.
\item
If $r$ is odd, then
$F_0(U) = 0$.
\end{enumerate}

It is shown in loc. cit. that
$H^1_{\Gal}(k, S) \cong  \Z/2 \oplus \Z/2$.
Hence the claim follows from 
Theorem \ref{rational} and the following theorem.

\begin{theorem}[Colliot-Th\'el\`ene, \cite{c-t}(4.7)]
\begin{enumerate}
\item
If $r$ is even and $\ord_k(a-b)=r$, then
$A_0(X) = 0$.
\item
If $r$ is even and $\ord_k(a-b)>r$, then
$A_0(X) \cong \Z/2$.
\item
If $r$ is odd, then
$A_0(X) \cong \Z/2 \oplus \Z/2$.
\end{enumerate}
\end{theorem}

One can extend this example to 
the case where $k(\sqrt{d})/k$ is a ramified extension,
by using a result of Dalawat \cite{dalawat}.
\end{example}

\begin{example}\label{cubic}
Let $a \in k^*$ and let $X$ be a cubic surface defined by
$T_0^3 + T_1^3 + T_2^3 + aT_3^3 = 0$ in $\mathbb{P}^3_k$.
Let $\zeta \in \bar{k}$ be a primitive cubic root of unity.
Set $r = \ord_k(a)$.
Let $C$ be a divisor on $X$, and put $U=X \setminus C$.
We assume that the irreducible components of 
$\bar{C}$ generate $NS(\bar{X})$,
and that $D(C)=0$.
(For example, we can take $C$ to be the union of the $27$ lines on $X$.
Cf. Remark \ref{rem-genus-zero})
We claim the following.

\begin{enumerate}
\item
If $p \not= 3$ and $r \equiv 0 \mod 3$, then
$F_0(U) \cong \Z/3 \oplus \Z/3$.
\item
If $p \not= 3, r \not\equiv 0 \mod 3$ and $\zeta \not\in k$, then
$F_0(U) \cong \Z/3$.
\item
Suppose $\zeta \in k$.
If $p \not= 3$, assume $r \not\equiv 0 \mod 3$.
If $p=3$, assume $r \equiv 1 \mod 3$.
Then
$F_0(X) = 0$.
\end{enumerate}

Manin \cite{manin2} has observed that
$H^1_{\Gal}(k, S) \cong  \Z/3 \oplus \Z/3$.
Hence the claim follows from 
Theorem \ref{rational} and the following theorem.

\begin{theorem}[Saito/Sato \cite{ss} (5.1.1)]
\begin{enumerate}
\item
If $p \not= 3$ and $r \equiv 0 \mod 3$, then
$A_0(X) = 0$.
\item
If $p \not= 3, r \not\equiv 0 \mod 3$ and $\zeta \not\in k$, then
$A_0(X) \cong \Z/3$.
\item
Suppose $\zeta \in k$.
If $p \not= 3$ assume $r \not\equiv 0 \mod 3$.
If $p=3$, assume $r \equiv 1 \mod 3$.
Then
$A_0(X) \cong \Z/3 \oplus \Z/3$.
\end{enumerate}
\end{theorem}
\end{example}

%
%
%
%
%
%
%
%
%
%


\begin{thebibliography}{99}
\bibitem{bloch} S. Bloch, 
On the Chow groups of certain rational surfaces. 
Ann. Sci. \'Ecole Norm. Sup. (4)  14  (1981), no. 1, 41--59.
\bibitem{bloch2} S. Bloch, 
Algebraic $K$-theory and classfield theory for arithmetic surfaces.  
Ann. of Math. (2)  114  (1981), no. 2, 229--265.
\bibitem{ct} J-L. Colliot-Th\'el\`ene, 
Hilbert's Theorem 90 for $K_{2}$, 
with application to the Chow groups of rational surfaces. 
Invent. Math. 71 (1983), no. 1, 1--20.
\bibitem{ct-survey} J-L. Colliot-Th\'el\`ene, 
Cycles alg\'ebriques de torsion et $K$-th\'eorie alg\'ebrique.
Arithmetic algebraic geometry (Trento, 1991),  1--49, 
Lecture Notes in Math., 1553, Springer, Berlin, 1993.
\bibitem{ct-survey2} J-L. Colliot-Th\'el\`ene, 
L'arithm\'etique du groupe de Chow des z\'ero-cycles.
J. Th\'eor. Nombres Bordeaux 7 (1995), no. 1, 51--73. 
\bibitem{cts} J-L. Colliot-Th\'el\`ene, S. Saito,
Z\'ero-cycles sur les vari\'et\'es $p$-adiques 
et groupe de Brauer.
Internat. Math. Res. Notices 1996, no. 4, 151--160.
\bibitem{ctsansuc2} J-L. Colliot-Th\'el\`ene, J-J. Sansuc,
La descente sur les vari\'et\'es rationnelles.
Journ\'ees de G\'eometrie Alg\'ebrique d'Angers, 
Juillet 1979/Algebraic Geometry, Angers, 1979,  
pp. 223--237, Sijthoff \& Noordhoff, 
Alphen aan den Rijn---Germantown, Md., 1980.
\bibitem{ctsansuc} J-L. Colliot-Th\'el\`ene, J-J. Sansuc,
On the Chow groups of certain rational surfaces: 
a sequel to a paper of S. Bloch. 
Duke Math. J. 48 (1981), no. 2, 421--447. 
\bibitem{c-t} D. Coray, M. Tsfasman,
Arithmetic on singular Del Pezzo surfaces.
Proc. London Math. Soc. (3) 57 (1988), no. 1, 25--87. 
\bibitem{dalawat} C. Dalawat, 
Le groupe de Chow d'une surface de Ch\^atelet sur un corps local.
Indag. Math. (N.S.)  11  (2000),  no. 2, 173--185. 
\bibitem{fv} E. Friedlander, V. Voevodsky,
Bivariant cycle cohomology. 
Cycles, transfers, and motivic homology theories, 138--187,
Ann. of Math. Stud., 143, Princeton Univ. Press, Princeton, NJ, 2000.
\bibitem{geisser1} T. Geisser,
Arithmetic homology and an integral version of Kato's conjecture
 J. Reine Angew. Math.  644  (2010), 1--22. 
\bibitem{geisser2} T. Geisser,
On Suslin's singular homology and cohomology. 
Doc. Math. 2010, Extra volume: Andrei A. Suslin sixtieth birthday, 223--249.
\bibitem{gl} T. Geisser, M. Levine, 
The Bloch-Kato conjecture and a theorem of Suslin-Voevodsky. 
J. Reine Angew. Math. 530 (2001), 55--103.
\bibitem{gs} P. Gille, T. Szamuely,
Central simple algebras and Galois cohomology. 
Cambridge Studies in Advanced Mathematics, 101. 
Cambridge University Press, Cambridge, 2006.
\bibitem{hira} T. Hiranouchi,
Class field theory for open curves over p-adic fields
Math. Z.  266  (2010),  no. 1, 107--113. 
\bibitem{js} U. Jannsen, S. Saito,
Kato homology of arithmetic schemes 
and higher class field theory over local fields. 
Kazuya Kato's fiftieth birthday. 
Doc. Math. 2003, Extra Vol., 479--538.
\bibitem{js2} U. Jannsen, S. Saito,
Bertini theorems and Lefschetz pencils over discrete valuation rings, 
with applications to higher class fileld theory,
to appear in J. of Algebraic Geometry.

\bibitem{jannsen} U. Jannsen,
Weights in arithmetic geometry. Jpn. J. Math. 5 (2010), no. 1, 73--102.

\bibitem{kahn} B. Kahn,
Relatively unramified elements in cycle modules. 
J. K-theory 7 (2011), 409-427.
\bibitem{kato} K. Kato,
A generalization of local class field theory by using $K$-groups I.
J. Fac. Sci. Univ. Tokyo, no. 26, (1979), 303--376.
\bibitem{ks} K. Kato, S. Saito, 
Unramified class field theory of arithmetical surfaces.  
Ann. of Math. (2)  118  (1983),  no. 2, 241--275. 
\bibitem{lichtenbaum} S. Lichtenbaum, 
Duality theorems for curves over $p$-adic fields.  
Invent. Math.  7  1969 120--136.
\bibitem{manin} Y. Manin,
  Le groupe de Brauer-Grothendieck en g\'eom\'etrie diophantienne.
  Actes du Congr\`es International des Math\'ematiciens (Nice, 1970), 
  Tome 1, pp. 401--411.
\bibitem{manin2} Y. Manin,
Cubic forms: algebra, geometry, arithmetic. 
North-Holland Mathematical Library, Vol. 4. 
North-Holland Publishing Co., Amsterdam-London; 
American Elsevier Publishing Co., New York, 1974.
\bibitem{mat} A. Mattuck, 
Abelian varieties over $p$-adic ground fields. 
Ann. Math. 62(2), 92--119 (1955). 
\bibitem{Merkurjev} A. Merkurjev,
On the torsion in $K_{2}$ of local fields. 
Ann. of Math. (2) 118 (1983), no. 2, 375--381.
\bibitem{mvw} C. Mazza, V. Voevodsky, C. Weibel,
Lecture notes on motivic cohomology. 
Clay Mathematics Monographs, 2. 
American Mathematical Society, Providence, RI; 
Clay Mathematics Institute, Cambridge, MA, 2006.
\bibitem{ps} R. Parimala, V. Suresh,
  Zero-cycles on quadric fibrations: 
  finiteness theorems and the cycle map,
  Invent. Math.  122  (1995),  no. 1, 83--117.
\bibitem{saito1} S. Saito,
Class field theory for curves over local fields. 
J. Number Theory 21 (1985), no. 1, 44--80.
\bibitem{saito} S. Saito,
  On the cycle map for torsion algebraic cycles of codimension two.
  Invent. Math.  106  (1991),  no. 3, 443--460.
\bibitem{ss} S. Saito, K. Sato, 
Zero-cycles on varieties over p-adic fields and Brauer groups.
arXiv:0906.2273.
\bibitem{sato2} K. Sato,
  Non-divisible cycles on surfaces over local fields.  
  J. Number Theory  114  (2005),  no. 2, 272--297.
\bibitem{svh} C. Scheiderer, J. van Hamel, 
Cohomology of tori over $p$-adic curves.  
Math. Ann.  326  (2003),  no. 1, 155--183.
\bibitem{schmidt} A. Schmidt,
Singular homology of arithmetic schemes.  
Algebra Number Theory  1  (2007),  no. 2, 183--222.
\bibitem{ss2} A. Schmidt, M. Spie\ss,
Singular homology and class field theory of varieties over finite fields.
J. Reine Angew. Math.  527  (2000), 13--36.
\bibitem{serre3} J.-P. Serre, 
Groupes alg\'ebriques et corps de classes. 
Hermann, 1959.
\bibitem{sj} A. Suslin, S. Joukhovitski, 
Norm varieties.  J. Pure Appl. Algebra  206  (2006),  no. 1-2, 245--276.
\bibitem{sv} A. Suslin, V. Voevodsky,
Singular homology of abstract algebraic varieties.
Invent. Math. 123 (1996), no. 1, 61--94. 
\bibitem{sv2} A. Suslin, V. Voevodsky,
Bloch-Kato conjecture and motivic cohomology with finite coefficients. 
The arithmetic and geometry of algebraic cycles (Banff, AB, 1998), 117--189,
NATO Sci. Ser. C Math. Phys. Sci., 548, Kluwer Acad. Publ., Dordrecht, 2000.
\bibitem{sz2} T. Szamuely,
Sur l'application de r\'eciprocit\'e 
pour une surface fibr\'ee en coniques d\'efinie sur un corps local.
$K$-Theory 18 (1999), no. 2, 173--179.
\bibitem{sz} T. Szamuely,
Sur la th\'eorie des corps de classes 
pour les vari\'et\'es sur les corps $p$-adiques.
J. Reine Angew. Math. 525 (2000), 183--212.
\bibitem{voevodsky} V. Voevodsky, 
Triangulated categories of motives over a field.  
In: Cycles, transfers, and motivic homology theories, 188--238, 
Ann. of Math. Stud., 143, Princeton Univ. Press, Princeton, NJ,  2000.
\bibitem{voevodsky2} V. Voevodsky, 
On motivic cohomology with $\mathbf Z/l$-coefficients. 
Ann. of Math. (2) 174 (2011), no. 1, 401--438.
\bibitem{weibel} C. Weibel,
The norm residue isomorphism theorem.
J. Topology  (2009)   2  (2):  346-372.
\bibitem{wiesend} G. Wiesend, 
Class field theory for arithmetic schemes.  
Math. Z.  256  (2007),  no. 4, 717--729.
(With an appendix by A. Schmidt.)
\bibitem{yama} T. Yamazaki,
On Chow and Brauer groups of a product of Mumford curves.
Math. Ann.  333  (2005),  no. 3, 549--567.
\bibitem{yama2} T. Yamazaki,
Class field theory for a product of curves over a local field. 
Math. Z. 261 (2009), no. 1, 109--121.
\end{thebibliography}
\end{document}